\documentclass{amsart}

\newtheorem{theorem}[equation]{Theorem}

\newtheorem{proposition}[equation]{Proposition}

\numberwithin{equation}{section}

\usepackage{amscd,amsmath,amssymb,verbatim}

\begin{document}

\title{Exponential sums and finite field $A$-hypergeometric functions}
\author{Alan Adolphson}
\address{Department of Mathematics\\
Oklahoma State University\\
Stillwater, Oklahoma 74078}
\email{adolphs@math.okstate.edu}
\date{\today}
\keywords{}
\subjclass{}
\begin{abstract}
We define finite field $A$-hypergeometric functions and show that they are Fourier expansions of families of exponential sums on the torus.  For an appropriate choice of $A$, our finite field $A$-hypergeometric function can be specialized to the finite field ${}_kF_{k-1}$-hypergeometric function defined by McCarthy.  
\end{abstract}
\maketitle


\section{Introduction}

Let  ${\mathbb F}_q$ be the finite field of $q$ elements, $q=p^a$, $p$ a prime.  Fix a nontrivial additive character $\Psi:{\mathbb F}_q\to{\mathbb C}$.  For a multiplicative character $\chi:{\mathbb F}_q^\times\to{\mathbb C}^\times$, denote by $g(\chi)$ the Gauss sum
\[ g(\chi) = \sum_{x\in{\mathbb F}_q^\times} \chi(x)\Psi(x)\in{\mathbb C}. \]
Let $\widehat{{\mathbb F}_q^\times}$ denote the group of multiplicative characters of ${\mathbb F}_q^\times$.  For $\chi\in\widehat{{\mathbb F}_q^\times}$ and ${\bf b}=(b_1,\dots,b_n)\in({\mathbb Z}/(q-1){\mathbb Z})^n$, let $\chi^{\bf b} = (\chi^{b_1},\dots,\chi^{b_n})\in(\widehat{{\mathbb F}_q^\times})^n$.

Fix a set $A=\{{\bf a}_1,\dots,{\bf a}_N\}\subseteq({\mathbb Z}/(q-1){\mathbb Z})^n$ and a parameter $\beta=(\beta_1,\dots,\beta_n)\in (\widehat{{\mathbb F}_q^\times})^n$.  Put 
\[ L_\beta = \bigg\{\chi = (\chi_1,\dots,\chi_N)\in (\widehat{{\mathbb F}_q^\times})^N\mid \prod_{i=1}^N \chi_i^{{\bf a}_i} = \beta\bigg\}. \]
Let $\lambda = (\lambda_1,\dots,\lambda_N)\in({\mathbb F}_q^\times)^N$.  Gel'fand, Graev, and Zelevinski\u{i}\cite[Proposition~1]{GGZ} give a formula for a formal series solution of an $A$-hypergeometric system.  By analogy with that formula, we define the associated finite field $A$-hypergeometric function $F_A(\beta;-):({\mathbb F}_q^\times)^N\to{\mathbb C}$ by the formula
\begin{equation}
F_A(\beta;\lambda) = (q-1)^{n-N}\sum_{\chi\in L_\beta} g(\bar{\chi}_1)\cdots g(\bar{\chi}_N)\chi_1(\lambda_1)\cdots\chi_N(\lambda_N).
\end{equation}
The factor $(q-1)^{n-N}$ is introduced to simplify the statement of Theorem 1.3 below.

The Gauss sums $g(\chi)$ depend on the choice of additive character~$\Psi$, hence so does $F_A(\beta;\lambda)$.  If $\Psi'$ is another nontrivial additive character, there exists $c\in{\mathbb F}_q^\times$ such that $\Psi'(x)=\Psi(cx)$ for all $x\in{\mathbb F}_q$. If we define $g'(\chi) = \sum_{x\in{\mathbb F}_q^\times}\chi(x)\Psi'(x)$, then the change of variable $x\mapsto c^{-1}x$ shows that $g'(\chi) = \bar{\chi}(c)g(\chi)$.  If we define $F'_A(\beta;\lambda)$ by replacing $g(\bar{\chi}_i)$ by $g'(\bar{\chi}_i)$ on the right-hand side of Eq.~(1.1), it follows that $F'_A(\beta;\lambda) = F_A(\beta; c\lambda)$.
In some cases one can normalize $F_A(\beta;\lambda)$ so that the normalized function is independent of the choice of additive character (see Section~3).

There is also a family of exponential sums associated to $A$ and $\beta$.  For ${\bf a}=(a_1,\dots,a_n)\in({\mathbb Z}/(q-1){\mathbb Z})^n$, write $x^{\bf a} = x_1^{a_1}\cdots x_n^{a_n}$.  For $\lambda = (\lambda_1,\dots,\lambda_N)\in({\mathbb F}_q^\times)^N$ define
\begin{equation}
S_A(\beta,\lambda) = \sum_{x=(x_1,\dots,x_n)\in({\mathbb F}_q^\times)^n} \bar{\beta}_1(x_1)\cdots \bar{\beta}_n(x_n)\Psi\bigg(\sum_{i=1}^N \lambda_ix^{{\bf a}_i}\bigg).
\end{equation}
The main result of this note is the following statement.
\begin{theorem}
For all $A$, $\beta$, and $\lambda$ as above, $S_A(\beta,\lambda) = F_A(\beta;\lambda)$. 
\end{theorem}

The $p$-adic version of the equality of Theorem 1.3, where all multiplicative characters are expressed as powers of the Teichm\"uller character, is similar to \cite[Eq.~(6.2)]{AS} and can be derived from it.  The results of \cite{AS} can be used to relate $F_A(\beta;\lambda)$ to mod~$p$ solutions of $A$-hypergeometric systems over ${\mathbb Q}$ and to truncated hypergeometric series.  We give a simple direct proof of Theorem~1.3 in Section~2.  In Section~3 we relate our definition to that given recently by McCarthy\cite{Mc}.

\section{Proof of Theorem 1.3}

Let ${\mathcal F}$ be the set of all ${\mathbb C}$-valued functions on $({\mathbb F}_q^\times)^N$.  Then ${\mathcal F}$ is a ${\mathbb C}$-vector space of dimension $(q-1)^N$.  We denote the trivial character on ${\mathbb F}_q^\times$ by $\varepsilon$.  It is easy to check that for $\chi=(\chi_1,\dots,\chi_N)\in(\widehat{{\mathbb F}_q^\times})^N$, 
\begin{equation}
\sum_{\lambda\in({\mathbb F}_q^\times)^N} \chi(\lambda) = \begin{cases} 0 & \text{if $\chi_i\neq\varepsilon$ for some $i$,} \\ (q-1)^N & \text{if $\chi_i=\varepsilon$ for all $i$.} \end{cases}
\end{equation}
This relation implies that the set $(\widehat{{\mathbb F}_q^\times})^N$ is an orthogonal basis for ${\mathcal F}$ relative to the inner product $(f,g) = \sum_{\lambda\in({\mathbb F}_q^\times)^N} f(\lambda)\overline{g(\lambda)}$.  We prove Theorem 1.3 by showing that $F_A(\beta;\lambda)$ is the Fourier expansion of $S_A(\beta,\lambda)$ relative to this orthogonal basis.  
Write
\begin{equation}
S_A(\beta,\lambda) = \sum_{\chi\in(\widehat{{\mathbb F}_q^\times})^N} c_\chi\chi(\lambda),
\end{equation}
where $c_\chi\in{\mathbb C}$.  Theorem~1.3 is an immediate consequence of the following result and the definition of $F_A(\beta;\lambda)$.

\begin{proposition}
With notation as above,
\[ c_\chi = \begin{cases} (q-1)^{n-N}g(\bar{\chi}_1)\cdots g(\bar{\chi}_N) & \text{if $\prod_{i=1}^N \chi_i^{{\bf a}_i} = \beta$,} \\ 0 & \text{if $\prod_{i=1}^N \chi_i^{{\bf a}_i} \neq \beta$.} \end{cases} \]
\end{proposition}

\begin{proof}
Fix $\rho=(\rho_1,\dots,\rho_N)\in(\widehat{{\mathbb F}_q^\times})^N$.  Multiplying both sides of (2.2) by $\bar{\rho}(\lambda)$, summing over $\lambda\in({\mathbb F}_q^\times)^N$, and using (2.1) and the fact that $\Psi$ is an additive character gives 
\begin{equation}
(q-1)^Nc_\rho = \sum_{\lambda\in({\mathbb F}_q^\times)^N}\sum_{x\in({\mathbb F}_q^\times)^n} \prod_{i=1}^n \bar{\beta}_i(x_i)\prod_{j=1}^N \bar{\rho}_j(\lambda_j)\Psi(\lambda_jx^{{\bf a}_j}).
\end{equation}
The change of variable $\lambda_j\to\lambda_j/x^{{\bf a}_j}$ gives
\[ \bar{\rho}_j(\lambda_j)\Psi(\lambda_jx^{{\bf a}_j}) \mapsto \rho_j^{{\bf a}_j}(x) \bar{\rho}_j(\lambda_j)\Psi(\lambda_j). \]
Substitution into (2.4) gives
\begin{align}
(q-1)^Nc_\rho &= \sum_{x\in({\mathbb F}_q^\times)^n} \bigg(\prod_{i=1}^N \rho_i^{{\bf a}_i}(x)\bigg) \bar{\beta}(x) \sum_{\lambda\in({\mathbb F}_q^\times)^N} \prod_{j=1}^N\bar{\rho}_j(\lambda_j)\Psi(\lambda_j) \\
 &= \prod_{j=1}^N g(\bar{\rho}_j) \sum_{x\in({\mathbb F}_q^\times)^n}\bigg(\prod_{i=1}^N \rho_i^{{\bf a}_i}(x)\bigg) \bar{\beta}(x). \nonumber
\end{align}
Analogous to (2.1) we have the relation
\begin{equation}
\sum_{x\in({\mathbb F}_q^\times)^n} \bigg(\prod_{i=1}^N \rho_i^{{\bf a}_i}(x)\bigg) \bar{\beta}(x) = \begin{cases} (q-1)^n & \text{if $\prod_{i=1}^N \rho_i^{{\bf a}_i} = \beta$,} \\ 0 & \text{if $\prod_{i=1}^N \rho_i^{{\bf a}_i} \neq \beta$.} \end{cases}
\end{equation}
Equations (2.5) and (2.6) imply the proposition.
\end{proof}

\section{Comparison with other definitions}

A finite field analogue of the classical ${}_kF_{k-1}$-hypergeometric function was first defined by Greene\cite{G}.  In recent work, McCarthy\cite{Mc} gives a somewhat different definition and compares his definition with that of Greene and with the definition of hypergeometric sum given by N. Katz\cite[Ch.~8.2]{K}.  In this section we show that for an appropriate choice of $A$, our finite field $A$-hypergeometric function specializes to that of McCarthy (up to a constant factor). 

We recall the definition of McCarthy\cite[Definition~1.4]{Mc}.  Let $\alpha_1,\dots,\alpha_{2k-1}\in\widehat{{\mathbb F}_q^\times}$ and let $t\in{\mathbb F}_q^\times$.  Define
\begin{multline}
{}_kF_{k-1}\bigg(\begin{matrix}\alpha_1 & \alpha_2 & \dots & \alpha_k \\ & \alpha_{k+1} & \dots & \alpha_{2k-1} \end{matrix}\; \bigg|\; t\bigg) = \\ 
\frac{1}{q-1} \sum_{\chi\in\widehat{{\mathbb F}_q^\times}}\bigg(\prod_{i=1}^k \frac{g(\alpha_i\chi)}{g(\alpha_i)} \prod_{j=k+1}^{2k-1}\frac{g(\bar{\alpha_j}\bar{\chi})}{g(\bar{\alpha}_j)}\bigg) g(\bar{\chi})\chi((-1)^k t).
\end{multline}
The Gauss sums $g(\alpha_i)$ and $g(\bar{\alpha}_j)$ in the denominators of the products make this expression independent of the choice of additive character $\Psi$.

Given a classical hypergeometric series, Dwork-Loeser\cite[Appendix]{DL} gives a choice of $A$ and $\beta$ for which the associated $A$-hypergeometric series specializes to the given hypergeometric series.  Following their suggestion, we take $n=2k-1$ and $A=\{{\bf a}_1,\dots,{\bf a}_{2k}\}$, where ${\bf a}_1,\dots,{\bf a}_{2k-1}$ are the standard unit basis vectors and
${\bf a}_{2k} = (1,\dots,1,-1,\dots,-1)$ ($1$ repeated $k$ times followed by $-1$ repeated $k-1$ times).  We take $\beta = (\bar{\alpha}_1,\dots,\bar{\alpha}_k,\alpha_{k+1},\dots,\alpha_{2k-1})$.  It is then straightforward to check that
\[ L_\beta = \{(\bar{\alpha}_1\bar{\chi},\dots,\bar{\alpha}_k\bar{\chi},\alpha_{k+1}\chi,\dots,\alpha_{2k-1}\chi,\chi)\mid \chi\in\widehat{{\mathbb F}_q^\times}\}. \]
Equation (1.1) then becomes
\begin{multline}
F_A(\beta;\lambda) = \frac{1}{q-1}\sum_{\chi\in\widehat{{\mathbb F}_q^\times}} \bigg(\prod_{i=1}^k g(\alpha_i\chi) \prod_{j=k+1}^{2k-1} g(\bar{\alpha}_j\bar{\chi})\bigg)g(\bar{\chi}) \\
\cdot\bigg(\prod_{i=1}^k \bar{\alpha}_i(\lambda_i)\bar{\chi}(\lambda_i) \prod_{j=k+1}^{2k-1} \alpha_j(\lambda_j)\chi(\lambda_j)\bigg) \chi(\lambda_{2k}).
\end{multline}
Making the specialization $\lambda_i\mapsto 1$ for $i=1,\dots,2k-1$ and $\lambda_{2k}\mapsto (-1)^kt$ gives
\begin{multline}
F_A(\beta;1,\dots,1,(-1)^kt) =  \\
\frac{1}{q-1}\sum_{\chi\in\widehat{{\mathbb F}_q^\times}} \bigg(\prod_{i=1}^k g(\alpha_i\chi) \prod_{j=k+1}^{2k-1} g(\bar{\alpha}_j\bar{\chi})\bigg)g(\bar{\chi})\chi((-1)^kt).
\end{multline}
Put $C = (\prod_{i=1}^kg(\alpha_i)\prod_{j=k+1}^{2k-1} g(\bar{\alpha}_j))$.  Comparing (3.3) with (3.1) shows that
\begin{equation}
C^{-1}F_A(\beta;1,\dots,1,(-1)^kt) = {}_kF_{k-1}\bigg(\begin{matrix}\alpha_1 & \alpha_2 & \dots & \alpha_k \\ & \alpha_{k+1} & \dots & \alpha_{2k-1} \end{matrix}\; \bigg|\; t\bigg),
\end{equation}
and the assertion of Theorem 1.3 is that
\begin{multline}
C\cdot{}_kF_{k-1}\bigg(\begin{matrix}\alpha_1 & \alpha_2 & \dots & \alpha_k \\ & \alpha_{k+1} & \dots & \alpha_{2k-1} \end{matrix}\; \bigg|\; t\bigg) = \\
\sum_{x\in({\mathbb F}_q^\times)^{2k-1}} \prod_{i=1}^k {\alpha}_i(x_i)\prod_{j=k+1}^{2k-1}\bar{\alpha}_j(x_j)\Psi\bigg(x_1 + \cdots + x_{2k-1} +
(-1)^kt\frac{x_1\cdots x_k}{x_{k+1}\cdots x_{2k-1}}\bigg).
\end{multline}

\end{document}